\documentclass{amsart}
\usepackage{amsmath}
\usepackage{amsthm}
\usepackage{amssymb}
\usepackage{amsfonts}
\usepackage{comment}
\usepackage{url}
\usepackage{xspace}
\usepackage{hyperref}
\newtheorem{PROP}{Proposition}[section]
\newtheorem{THM}[PROP]{Theorem}
\newtheorem{LM}[PROP]{Lemma}
\newtheorem{COR}[PROP]{Corollary}
\theoremstyle{definition}
\newtheorem{DEF}[PROP]{Definition}
\newtheorem{EXA}[PROP]{Example}
\newtheorem{REM}[PROP]{Remark}

\newtheorem{PROB}[PROP]{Problem}

\newcommand{\uv}[1]{``#1''}
\newcommand{\abs}[1]{\lvert#1\rvert}
\newcommand{\ol}[1]{\ensuremath{\overline{#1}}}
\newcommand{\ul}[1]{\ensuremath{\underline{#1}}}
\newcommand{\limti}[1]{\lim\limits_{#1\to\infty}}

\newcommand{\dcc}[1]{\ensuremath{\lfloor #1 \rfloor}} 
\newcommand{\enu}{\renewcommand{\theenumi}{\roman{enumi}}\renewcommand{\labelenumi}{{\rm(\theenumi)}}}
\newcommand{\RR}{\ensuremath{\mathbb R}\xspace}
\newcommand{\NN}{\ensuremath{\mathbb N}\xspace}
\newcommand{\ve}{\ensuremath{\varepsilon}\xspace}
\newcommand{\vr}{\ensuremath{\varrho}\xspace}

\newcommand{\intrv}[2]{\ensuremath{[#1,#2]}}
\newcommand{\intrvr}[2]{\ensuremath{]#1,#2]}}

\newcommand{\intrvo}[2]{\ensuremath{]#1,#2[}}

\newcommand{\post}[2]{\ensuremath{(#1)_{#2=1}^\infty}}

\newcommand{\alp}[1]{\alpha(#1)}
\newcommand{\bet}[1]{\beta(#1)}
\includecomment{mine}
\begin{document}
\title{Gaps and the exponent of convergence of an integer sequence}
\author{Georges Grekos}
\address{Universit\'e Jean Monnet, Facult\'e des Sciences et Techniques, D\'epartement de Math\'ematiques, 23 Rue du Docteur Paul Michelon, 42023 Saint-Etienne Cedex 2, France}
\email{\tt grekos@univ-st-etienne.fr}
\author{Martin Sleziak}
\author{Jana Tomanov{\'a}}
\address{Department of Algebra, Geometry and Mathematical Education, Faculty of Mathematics, Physics and Informatics, Comenius University, Mlynsk\'a dolina, 842 48 Bratislava, Slovakia}
\email{\tt sleziak@fmph.uniba.sk, tomanova@fmph.uniba.sk}

\begin{abstract}
Professor Tibor \v{S}al\'at, at one of his seminars at Comenius
University, Bratislava, asked to study the influence of gaps
of an integer sequence $A=\{a_1<a_2<\dots<a_n<\dots\}$ on its
exponent of convergence. The exponent of convergence of $A$
coincides with its upper exponential density. In this paper we
consider an extension of Professor \v{S}al\'at's question and we
study the influence of the sequence of ratios
$\post{\frac{a_m}{a_{m+1}}}m$ and of the sequence
$\post{\frac{a_{m+1}-a_m}{a_{m}}}m$ on the upper and on the lower
exponential densities of $A$.
\end{abstract}
\maketitle
\section{Introduction}
The concept of exponent of convergence is introduced in
\cite{POLYASZEGO1NEW}. The authors of this treatise proved that
for any real sequence $r=\post {r_n}n$, $0<r_1 \leq r_2 \leq
\ldots \leq r_n \leq \ldots,$ with $\limti n r_n = +\infty$, there
exists $\tau\in\intrv 0{+\infty}$, such that the series
$\sum\limits_{n=1}^\infty r_n^{-\sigma}$ is convergent whenever
$\sigma>\tau$ and divergent whenever $\sigma<\tau$ (\cite[Part I,
Exercises 113,114]{POLYASZEGO1NEW}). The number $\tau$ is called
the \emph{exponent of convergence} of the sequence $r$ and denoted
by $\tau(r)$. The exponent of convergence of real non-decreasing
sequences was also studied in
\cite{KOSTYR1979,KOSSAL1982,SALAT1984}.
It was proved by P\'olya and Szeg\"o \cite[Part I, Exercises 113,114]{POLYASZEGO1NEW}
that $\tau(r)$ can be calculated by the formula
\begin{equation}\label{EQD2}
\tau(r)=\limsup_{n\to\infty} \frac{\log n}{\log r_n}.
\end{equation}

In particular, if $r$ is an \emph{integer sequence} $A=\{a_1<a_2<\ldots<a_n<\ldots\}$
(that is, $A$ is an infinite subset of
$\NN=\{1,2,\dots\}$), then $A$ has an exponent of convergence $\tau(A) \in \intrv01$.

This simple observation indicates that when dealing with sequences of positive integers,
then the exponent of convergence could be related to the number-theoretic densities.

We recall the notion of exponential density  \cite{GREDENSSURV,LORENTZ}.

\begin{DEF}
The \emph{upper} and \emph{lower exponential densities} of an infinite subset
$A$ of $\NN$ are defined by
\begin{gather*}
\ol\ve(A)=\limsup_{k\to\infty} \frac{\log A(k)}{\log k},\\
\ul\ve(A)=\liminf_{k\to\infty} \frac{\log A(k)}{\log k},
\end{gather*}
where $A(x)$ denotes $\abs{A\cap\intrv1x}$.

If $\ol\ve(A)=\ul\ve(A)$, then we say that $A$ has the
exponential density $\ve(A)=\ol\ve(A)=\ul\ve(A)$.
\end{DEF}

One can easily see that, for an infinite subset
$A=\{a_1<a_2<\ldots<a_n<\ldots\}$ of $\NN$,
we have $\ol\ve(A)=\tau(A)=\limsup_{n\to\infty} \frac{\log n}{\log a_n}$, and
$\ul\ve(A)=\liminf_{n\to\infty} \frac{\log n}{\log a_n}$.



The purpose of this paper is the investigation of the influence of gaps
$$g_n=a_{n+1}-a_n$$
in the set $A=\{a_1<a_2<\dots<a_n<\dots\}\subseteq\NN$ on its exponent of convergence.
This study was suggested to us by the late
Professor Tibor \v{S}al\'at.
We will be also concerned with a slightly more general question about
the influence of gaps in $A$ on both exponential densities.

\section{Ratios of consecutive terms and exponent of convergence}

We are interested in the influence of gaps
$$g_n=a_{n+1}-a_n$$
in the set $A$ on the exponent of convergence. Since a gap of given length
has less influence if it is situated far from the origin, we might expect that (at least to
some extent) we can describe the behavior of the exponent of convergence in terms of the
asymptotic behavior of the fractions
$$\frac{g_n}{a_{n+1}}\qquad\text{ or }\qquad \frac{a_n}{a_{n+1}}.$$
Note that $\frac{a_n}{a_{n+1}}+\frac{g_n}{a_{n+1}}=1$.

\begin{DEF}
We define the upper and the lower \emph{limit ratios} of $A$ by
\begin{gather*}
\ol\vr(A)=\limsup_{n\to\infty} \frac{a_n}{a_{n+1}},\\
\ul\vr(A)=\liminf_{n\to\infty} \frac{a_n}{a_{n+1}}.
\end{gather*}
\end{DEF}

We remark that several related concepts have been studied in various contexts. The gap density $\lambda
(A)=\limsup\frac{a_{n+1}}{a_n}$ was introduced in \cite{GREVOL} to study properties of the
density set $D(A)=\{(\ul d(B),\ol d(B)); B\subseteq A\}$ of $A$, and further studied in
\cite{GRESATOGAPS}. Clearly $\ul\vr(A)=\frac1{\lambda(A)}$ (using the convention $\frac1\infty=0$).

%
%
%
%

The sets $A$ with $\ol\vr(A)=0$ (called thin sets) play a role in
the study of measures which can be regarded as certain extensions of
asymptotic density \cite{BFPR}.
The sets with $\ol\vr(A)<1$ (called almost thin sets) are studied in
connection with some ultrafilters on $\NN$ \cite{FLASTHIN}.

First we show that the exponent of convergence of any set $A\subseteq\NN$ with $\ol\vr(A)<1$ is equal to zero.
We will need the following well-known result (see e.g.~\cite[Problem 2.3.11]{KACZORNOWAK1},
\cite{STOLZ}):
\begin{THM}[Stolz Theorem]\label{THMSTOLZ}
Let $\post{x_n}n$, $\post{y_n}n$ be sequences of real numbers such
that $\post{y_n}n$ is strictly increasing, unbounded and
$$\limti n\frac{x_{n+1}-x_{n}}{y_{n+1}-y_{n}} = g,$$
$g\in\RR$. Then
$$\limti n \frac{x_n}{y_n}=g.$$
\end{THM}

\begin{PROP}\label{PROPVRTAU}
If $\ol\vr(A)<1$, then $\tau(A)=0$.
\end{PROP}

\begin{proof}
We will apply Theorem \ref{THMSTOLZ} to the sequences $x_n=\log n$
and $y_n=\log a_n$. Clearly, $y_n$ is strictly increasing and
unbounded. Note that $\limti n (x_{n+1}-x_n)=\limti n
\log\frac{n+1}n = 0$. Therefore
\begin{equation}\label{EQRATIO1}
 \limti n\frac{x_{n+1}-x_{n}}{y_{n+1}-y_{n}}=\limti n \frac{\log\frac{n+1}n}{\log\frac{{a_{n+1}}}{{a_n}}}=0,
\end{equation}
whenever $\log{\frac{a_{n+1}}{a_n}}$ is bounded from zero.
Thus the assumption $\ol\vr(A)<1$ is sufficient to infer this.

From (\ref{EQRATIO1}) we get $$\limti n \frac{\log n}{\log a_n}=0$$ by Stolz theorem. Thus $\tau(A)=0$.
\end{proof}

It remains only to analyse the case $\ol\vr(A)=1$.
The following examples show that
in this case nothing can be said about $\tau(A)$ in general.

\begin{EXA}\label{EXA1}
Let $a\in\intrvr01$, and let $A=\{\dcc{n^{\frac1a}}; n\in\NN\}$. Then
$\vr(A)=1$ and $\ve(A)=a$.
\end{EXA}

\begin{EXA}\label{EXA2}
Let $A=\{2^n;n\in\NN\}\cup\{2^{2^N}+1;N\in\NN\}$. Then
$\ol\vr(A)=1$ and $\ve(A)=0$.
\end{EXA}

\begin{EXA}\label{EXA3}
Let $A=\{a_n=\dcc{u_n}; n\ge 1,
u_n=\left(1+\frac1{\sqrt{n}}\right)^n\}$. Then $\vr(A)=1$ and
$\ve(A)=0$.
\end{EXA}

\begin{proof}
First observe that $u_n$ tends to infinity as $n\to\infty$, since
$$\limti n \log u_n= \limti n n \log \left(1+\frac1{\sqrt{n}}\right)=+\infty.$$
This yields that $a_n\sim u_n$ and $\log a_n \sim \log u_n$.

We have
$$\limti n \frac{\log n}{\log a_n}=\limti n \frac{\log n}{\log u_n}= \limti n \frac{\log n}{n\log \left(1+\frac1{\sqrt{n}}\right)}=0$$
and so $\tau(A)=\ve(A)=0$.

We will show that $\limti n \frac{u_n}{u_{n+1}}=1$, which implies $\vr(A)=1$.
Note that
$$\frac{u_n}{u_{n+1}}=\frac{t_n}{1+\frac1{\sqrt{n+1}}},$$
where
$$t_n=\left(\frac{1+\frac1{\sqrt{n}}}{1+\frac1{\sqrt{n+1}}}\right)^n.$$
So it is sufficient to show that $\limti n t_n = 1$
or equivalently $\limti n \ln t_n=0$.

Indeed
\begin{multline*}
\ln t_n=
n\left(\ln\left(1+\frac1{\sqrt{n}}\right)-\ln\left(1+\frac1{\sqrt{n+1}}\right)\right)
=\\=
n\left(\frac1{\sqrt{n}}-\frac1{2n}(1+o(1))-\frac1{\sqrt{n+1}}+\frac1{2(n+1)}(1+o(1))\right)
=\\=
\frac{n}{\sqrt{n}\sqrt{n+1}(\sqrt{n}+\sqrt{n+1})}+\frac12\left(\frac{n}{n+1}-1\right)+o(1)
\end{multline*}
tends to $0$ as $n\to\infty$.

Thus it remains to prove that $u_{n+1}-u_n>1$ for all sufficiently large $n$,
and hence the elements (integers) $a_n$ of the set $A$ are pairwise
distinct for $n$ large enough. We shall prove that $\limti n
(u_{n+1}-u_n)=+\infty$ 
using the function
$$f(x)=\left(1+\frac1{\sqrt{x}}\right)^x=e^{x\ln\left(1+\frac1{\sqrt{x}}\right)}.$$
An easy computation gives
\begin{equation}\label{EQNOTES1}
    f'(x)=f(x)\left(\ln\left(1+\frac1{\sqrt{x}}\right)-\frac1{2(\sqrt{x}+1)}\right)\ge \frac{f(x)}{2(\sqrt{x}+1)}.
\end{equation}
From the inequality $\left(1+\frac1t\right)^t\ge 2$, which is
valid for $t\ge1$, we get $f(x)\ge 2^{\sqrt{x}}$ and
\begin{equation}\label{EQNOTES1b}
    f'(x)\ge \frac{2^{\sqrt{x}}}{2(\sqrt{x}+1)}
\end{equation}
for $x\ge1$.
This implies that $\limti x f'(x)=+\infty$ and
$$\limti n (u_{n+1}-u_n)=\limti n (f(n+1)-f(n))=+\infty$$
by the mean value theorem.
\end{proof}

The above examples suggest the following question:
\begin{PROB}\label{PROBEPSRHO}
Given $a,b,c\in\intrv01$, $a\leq b$, does there exist
a subset $A\subseteq\NN$ such that $\ul\ve(A)=a$, $\ol\ve(A)=b$,
$\ul\vr(A)=c$ and $\ol\vr(A)=1$?
\end{PROB}


\section{Rate of proximity of $\frac{a_n}{a_{n+1}}$ to $1$ and the exponential
densities}\label{SECTIONGAPS}

The following three examples provide a motivation for the questions studied in this section.

\begin{EXA}\label{EXAG}
Let $a_n=\alpha b^n(1+o(1))$, for some $\alpha>0$ and $b>1$ (geometric-like sequence). Then
$\tau(A)=0$ and
$$\limti n\frac{g_n}{a_{n+1}}=\limti n 1-\frac{a_n}{a_{n+1}}=1-\frac1b\in \intrvo 01.$$
\end{EXA}

\begin{EXA}\label{EXAA}
If $a_n=k+ln$ is an arithmetic sequence then $\tau(A)=1$ and
$$\frac{g_n}{a_{n+1}}=\frac1{n}(1+o(1)).$$
\end{EXA}

\begin{EXA}\label{EXAP}
If $a_n=k+ln+tn^2$, then $\tau(A)=\frac12$ and
$$\frac{g_n}{a_{n+1}}=\frac2{n}(1+o(1)).$$

More generally, for any real $d>1$ and any integer sequence
$\post{a_n}n$ satisfying $a_n=tn^d(1+o(1))$ we get
$\tau(A)=\frac1d$ and
$$\frac{g_n}{a_{n+1}}=\frac{d}{n}(1+o(1)).$$
Thus for a real number $d>0$ and $\post{a_n}n$ as above, we have
$\tau(A)=\max\{1,\frac1d\}$.
\end{EXA}

The following corollary is a straightforward consequence of Proposition \ref{PROPVRTAU}.
\begin{COR}\label{COR3}
If $\liminf\limits_{n\to\infty} \frac{g_n}{a_{n+1}}>0$ then
$\tau(A)=\ve(A)=0$.
\end{COR}

This can be improved as follows.

\begin{PROP}\label{PROP3}
If $\liminf\limits_{n\to\infty} \frac{g_n}{a_{n}}>0$ then
$\tau(A)=\ve(A)=0$.
\end{PROP}

\begin{proof}
The hypothesis $\liminf\limits_{n\to\infty} \frac{g_n}{a_{n}}>0$
guarantees that there exists a $\delta>0$ such that
$\frac{g_n}{a_n}\ge\delta$ for each $n$. This implies
$$\frac{a_{n+1}}{a_n} = 1+\frac{g_n}{a_n} \ge 1+\delta.$$
Hence
$$\log a_n \ge \log c + n \log (1+\delta), \text{ for some constant }c.$$
From this we get
$$0 \le \limti n \frac{\log n}{\log a_n} \le \limti n \frac{\log
n}{\log c+n\log(1+\delta)} = 0.$$
\end{proof}

Examples \ref{EXAG}, \ref{EXAA} and \ref{EXAP} show
that there is a relation between the exponent of convergence and
the limit behavior of $\frac{g_n/a_{n+1}}{1/n}$. This is generalized in Proposition
\ref{PROPSTOLZMS}.

We will need a slightly more general form of Stolz Theorem. 
For the sake of completeness we include the proof of this result.
\begin{LM}\label{LMSTOLZ}
Let $\post{x_n}n$ and $\post{y_n}n$ be sequences of positive real
numbers such that $\post{y_n}n$ is strictly increasing and
unbounded. Then
$$
\liminf_{n\to\infty}\frac{x_{n+1}-x_n}{y_{n+1}-y_n}\le
\liminf_{n\to\infty}\frac{x_n}{y_n}\le
\limsup_{n\to\infty}\frac{x_n}{y_n}\le
\limsup_{n\to\infty}\frac{x_{n+1}-x_n}{y_{n+1}-y_n}.
$$
\end{LM}

\begin{proof}
Put $l=\liminf\limits_{n\to\infty}
\frac{x_{n+1}-x_n}{y_{n+1}-y_n}$. Given $\ve>0$, there exists $n_0$
such that $\frac{x_{n+1}-x_n}{y_{n+1}-y_n}\geq l-\ve$ for all $n>n_0$.
Using this fact we get, for all $n>n_0$,
$$x_n-x_{n_0}\ge (y_n-y_{n_0})(l-\ve).$$
Thus
$$\liminf\limits_{n\to\infty} \frac{x_n}{y_n} \ge l-\ve.$$
Since this is true for any $\ve>0$, we get
$$\liminf_{n\to\infty} \frac{x_n}{y_n} \ge l.$$

The proof of the second part of this lemma is analogous.
\end{proof}

Using the well-known equation $\sum_{k\leq n} \frac 1k = \ln n + \gamma +
o\left(\frac1n\right)$
we get the alternative formulae for exponential densities
\begin{gather}\label{EQD3}
\ol\ve(A)=\limsup_{n\to\infty} \frac{\sum_{k\leq n} \frac1k}{\sum_{k\leq a_n} \frac 1k},\\
\ol\ve(A)=\liminf_{n\to\infty} \frac{\sum_{k\leq n} \frac1k}{\sum_{k\leq a_n} \frac 1k}.
\end{gather}

\begin{PROP}\label{PROPSTOLZMS}
Let $A=\{a_1<a_2<\ldots<a_n<\ldots\}$ and
\begin{gather*}
\alp{A}=\liminf_{n\to\infty} n\frac{g_n}{a_{n+1}},\\
\bet{A}=\limsup_{n\to\infty} n\frac{g_n}{a_{n}}.
\end{gather*}
Then
\begin{equation}\label{EQSTOLZMS}
    \frac 1{\bet{A}} \leq \ul\ve(A)\leq \ol\ve(A)\leq \frac1{\alp{A}}.
\end{equation}
$($We use the convention $\frac1{\infty}=0$ and
$\frac10=\infty$.$)$

In particular, if $\alp A=\bet A$, then $\ve(A)=\frac1{\alp A}=\frac1{\bet A}$.
\end{PROP}

\begin{proof}
Applying Lemma \ref{LMSTOLZ} to (\ref{EQD3}) we get
$$\limsup\frac{\sum_{k\leq n}\frac 1k}{\sum_{k\leq a_n}\frac 1k}\leq \limsup \frac{\frac1{n+1}}{\sum_{a_n<k\leq a_{n+1}}\frac1k}
\leq \limsup\frac{\frac1{n+1}}{\frac{a_{n+1}-a_n}{a_{n+1}}} = \limsup \frac1{n+1} \cdot
\frac1{\frac{g_n}{a_{n+1}}}.
$$
By a similar argument we get
$$\liminf_{n\to\infty}\frac{\sum_{k\leq n}\frac 1k}{\sum_{k\leq a_n}\frac 1k} \geq \liminf_{n\to\infty} \frac{\frac1{n+1}}{\sum_{a_n<k\leq a_{n+1}}\frac1k} \geq
\liminf_{n\to\infty}\frac{\frac1{n+1}}{\frac{a_{n+1}-a_n}{a_{n}}}
= \liminf_{n\to\infty}\frac1{n+1}\frac{1}{\frac{g_n}{a_{n}}}.
$$
\end{proof}

The above lemma yields
\begin{THM}\label{THMGREKOS}
Let $\alpha$, $\beta$ be real numbers with $1\leq\alpha\leq\beta$.
\begin{enumerate}
\enu
  \item If there exists a real sequences $\post {e_n}n$ tending to zero such that, for each $n\geq 1$,
    $$(1+e_n)\frac\alpha n \leq \frac{g_n}{a_{n+1}},$$
    then
    $$\ol\ve(A) \leq \frac 1\alpha.$$
  \item If there exists a real sequences $\post {f_n}n$ tending to zero such that, for each $n\geq 1$,
    $$\frac{g_n}{a_{n}}\leq \frac\beta n (1+f_n),$$
    then
    $$\frac1\beta \leq \ul\ve(A).$$
\end{enumerate}
\end{THM}

\begin{EXA}
Proposition \ref{PROPSTOLZMS} can be used to compute the exponential density for some sets from the above examples.
By a straightforward computation we get the following values.
\begin{center}
\begin{tabular}{|l|c|c|}
  \hline
  $A=\{a_n; n\in\NN\}$ & $\alp A$ & $\bet A$ \\ \hline
  $A=\{\dcc{n^{\frac1a}}; n\in\NN\}$ & $\frac1a$ & $\frac1a$ \\ \hline
  $A=\{2^n;n\in\NN\}\cup\{2^{2^N}+1;N\in\NN\}$ & $0$ & $+\infty$ \\ \hline
  $A=\{n^2;n\in\NN, n \text{ is not square}\}$ & 2 & 4 \\\hline
  $a_n=\alpha b^n(1+o(1))$ & $+\infty$ & $+\infty$ \\ \hline
  $a_n=k+ln$ & $1$ & $1$ \\ \hline
  $a_n=k+ln+tn^2$ & $2$ & $2$ \\ \hline
  $a_n=tn^d(1+o(1))$ & $d$ & $d$ \\ \hline
\end{tabular}
\end{center}
We have shown in Example \ref{EXA2} that $\ve(A)=0$ for the set in the second row from this table.
It can be shown easily that $\ve(A)=\frac12$ for the set in the fourth row. These two examples show that
the inequalities in \eqref{EQSTOLZMS} can be strict.

In all other rows we have $\alp A=\bet A$, and the value of $\ve(A)$ can be computed using Proposition \ref{PROPSTOLZMS}.

The computation for the set in Example \ref{EXA3} is slightly more complicated.
We use the function $f(x)=\left(1+\frac1{\sqrt{x}}\right)^x$ again. Notice that \eqref{EQNOTES1} implies
that this function is increasing.

Using the mean value theorem and \eqref{EQNOTES1} we get
\begin{multline*}
n\frac{f(n+1)-f(n)}{f(n)}\ge
 n\frac{\min\{f'(c); c\in\intrv n{n+1}\}}{f(n)} \ge\\\ge
 n\frac{\min\{\frac{f(c)}{2(\sqrt{c}+1)}; c\in\intrv n{n+1}\}}{f(n)}
 \ge \frac{n}{f(n)} \cdot \frac{f(n)}{2(\sqrt{n+1}+1)}
\frac{n}{2(\sqrt{n+1}+1)}
\end{multline*}
and the last expression tends to $+\infty$ as $n\to\infty$.

We found out that $\alp A = \bet A = +\infty$. Thus
$\ve(A)=0$ by Proposition \ref{PROPSTOLZMS}.
\end{EXA}

\begin{REM}
If the set $A$ satisfies the hypotheses of Theorem \ref{THMGREKOS}
with $\alpha=\beta$, then $\ve(A)=\tau(A)=\frac1\alpha$. This
type of sets $A$ generalizes the sets considered in Example \ref{EXAA}
(sequences increasing arithmetically; $\alpha=\beta=1$) and in
Example \ref{EXAP} (sequences increasing polynomially;
$\alpha=\beta>1$).
\end{REM}

Now we state two refinements of Theorem
\ref{THMGREKOS}, Theorems \ref{THMGREKOS1} and \ref{THMGREKOS2}.

\begin{LM}\label{LMLOG}
For any $x\in\intrv0{\frac12}$ the inequality
$$-\ln(1-x)\le x+x^2$$
holds.
\end{LM}

\begin{proof}
By studying $f(x)=x+x^2-[-\ln(1-x)]$.
\end{proof}

\begin{THM}\label{THMGREKOS1}
Let $\beta$ be a real number, $\beta\geq1$.
If there exists a real sequence $\post {f_n}n$ tending to zero
such that, for each $n\geq 1$,
    $$\frac{g_n}{a_{n+1}}\leq \frac\beta n (1+f_n),$$
    then
    $$\ul\ve(A) \ge \frac1\beta.$$
\end{THM}

\begin{proof}
By the definition of $\ul\ve(A)$ it is sufficient to show that for each
$0<\ve<\frac1\beta$ there exists $n_0=n_0(\ve)$ such that for each $n\ge n_0$
$$\frac{\ln{n}}{\ln{a_n}} \ge \frac1\beta-\ve.$$
We will verify that
$$\ln a_{n+1} \le \frac{\ln{(n+1)}}{\frac1\beta-\ve}$$
for every $n$ large enough.

The hypothesis gives
\begin{gather*}
\frac{a_{n+1}-a_n}{a_{n+1}}\le\frac\beta{n}(1+f_n),\\
\frac{a_{n+1}}{a_n} \le \frac1{1-\frac\beta{n}(1+f_n)}.
\end{gather*}
Let us choose $n_1$ such that $\abs{f_n}<\frac12$, for each $n\ge n_1$. Then we get
$$\ln a_{n+1} = \ln a_{n_1} + \sum_{i=n_1}^n \ln\frac{a_{i+1}}{a_i} \le
\ln a_{n_1}-\sum_{i=n_1}^n \ln(1-\frac\beta{n}(1+f_n)).$$ Let
$n_2\ge n_1$ be such that $\abs{\frac\beta{n}(1+f_n)}\le\frac12$ whenever
$n\ge n_2$. Then we get
$$\ln a_{n+1} \le c_1 + \sum_{i=n_2}^n \frac\beta{i}(1+f_i)+
\sum_{i=n_2}^n \frac{\beta^2}{i^2}(1+f_i)^2.$$ Put
$\delta=\frac12\frac{\beta\ve}{1-\beta\ve}>0$ and choose $n_3\ge
n_2$ such that $\abs{f_i}\le\delta$, whenever $i\ge n_3$.

Then, for all $n\ge n_3$,
\begin{multline*}
\ln a_{n+1} \le c_2 + \sum_{i=n_3}^n \frac{\beta(1+\delta)}i +
\sum_{i=n_3}^n \frac{\beta^2(1+\delta)^2}{i^2} \le\\\le
c_3+\beta(1+\delta)\sum_{i=1}^n\frac1i \le
c_3+\beta(1+\delta)+\beta(1+\delta)\ln(n+1)
\end{multline*}
holds.

Obviously
$$\beta(1+\delta)<\frac1{\frac1\beta-\ve}$$
and therefore,
the right hand side of the above inequality is at most
$\frac{\ln(n+1)}{\frac1\beta-\ve}$ for every sufficiently large
$n$.
\end{proof}

\begin{THM}\label{THMGREKOS2}
Let $\alpha$ be a real number, $\alpha\ge1$. Suppose that for all
$n$
$$\frac{g_n}{a_n} \ge \frac\alpha{n}(1+e_n),$$
where $\limti n e_n=0$. Then $$\ol\ve(A)\le\frac1\alpha.$$
\end{THM}

Note that Proposition \ref{PROP3} can be deduced from the above theorem.

\begin{proof}
If $\alpha=1$ we are done. Suppose
that $\alpha>1$.

By the definition of $\ol\ve(A)$ it suffices to show that 
for each $\ve>0$ there is
an $n_0=n_0(\ve)$ such that
$$\frac{\ln n}{\ln a_n} \le \frac1\alpha+\ve,$$
whenever $n\ge n_0$. That is,
$\ln{n}\le\left(\frac1\alpha+\ve\right)\ln{a_n}$ or
equivalently $$\ln{a_n}\ge\frac{\ln{n}}{\frac1\alpha+\ve}.$$

Fix $\ve>0$.
We shall prove that for all $n$ sufficiently large
\begin{equation}\label{EQ1G1}
    \ln a_{n+1} \ge \frac{\ln(n+1)}{\frac1\alpha+\ve}.
\end{equation}
Let $\delta=\frac12\frac{\alpha\ve}{1+\alpha\ve}$. Obviously
$\delta<1$. Choose $n_1$ such that $\abs{e_n}\le\delta$ for
all $n\ge n_1$. The hypothesis of the theorem implies that for all $n$
$$\frac{a_{n+1}}{a_n} \ge 1+\frac\alpha{n}(1+e_n).$$
Then for $n\ge n_1$, we have
$$\ln a_{n+1} = \ln a_{n_1}+ \sum_{i=n_1}^n \ln\frac{a_{i+1}}{a_i} \ge
\ln a_{n_1}+ \sum_{i=n_1}^n
\ln\left(1+\frac\alpha{i}(1+e_i)\right),$$ and by the inequality
$$\ln(1+x)\ge x-\frac{x^2}2$$
(valid for $x\ge0$), we have
\begin{multline*}
\ln a_{n+1} \ge \ln a_{n_1} + \sum_{i=n_1}^n
\frac\alpha{i}(1+e_i) - \frac12 \sum_{i=n_1}^n
\frac{\alpha^2}{i^2}(1+e_i)^2 \ge\\\ge \ln a_{n_1} +
\alpha(1-\delta) \sum_{i=n_1}^n \frac1i -\frac12
(1+\delta)^2\alpha^2 \sum_{i=1}^\infty \frac1{i^2}.
\end{multline*}
Finally
$$\ln a_{n+1}\ge c_4 +\alpha(1-\delta)\sum_{i=1}^n\frac1i,$$
where $c_4$ does not depend on $n$. The last inequality implies
that
$$\ln a_{n+1} \ge c_4 + \alpha(1-\delta)\ln(n+1),$$
since $\sum_{i=1}^n\frac1i>\ln(n+1)$.

Now to deduce that \eqref{EQ1G1} is valid for all $n$ large
enough, it suffices to verify that
$\alpha(1-\delta)>\frac1{\frac1\alpha+\ve}$, which is
straightforward.
\end{proof}

Note that by Proposition \ref{PROP3} the study of exponential densities is non-trivial
only for sets $A$ such that $\liminf\limits_{n\to\infty}\frac{g_n}{a_n}=0$.
In view of this, the results stated in Theorems \ref{THMGREKOS1} and \ref{THMGREKOS2} are far from being complete since
only comparison of $\frac{g_n}{a_n}$ or $\frac{g_n}{a_{n+1}}$ to
sequences of the type \uv{constant times $\frac1n$} was
considered. Nevertheless these cases, motivated by Examples
\ref{EXAA} and \ref{EXAP}, are the most important ones as the following
observations show.

\textbf{Observation 1.} If $\frac{g_n}{a_n}$ is approximatively
$\frac1{n^\alpha}$ with $\alpha>1$, then $a_{n+1}$ is
approximatively $a_n(1+n^\alpha)$ and $a_n$ is approximatively
$c\prod\limits_{i=1}^n\left(1+\frac1{i^\alpha}\right)$. The
product $\prod\limits_{i=1}^\infty\left(1+\frac1{i^\alpha}\right)$
being convergent, we get that the set
$\{\dcc{c\prod\limits_{i=1}^n\left(1+\frac1{i^\alpha}\right)};n\in\NN\}$
is finite.

\textbf{Observation 2.} If $\frac{g_n}{a_n}$ is approximatively
$\frac1{n^\alpha}$ with $0<\alpha<1$, again $a_n$ would be close
to $u_n:=c\prod\limits_{i=1}^n\left(1+\frac1{i^\alpha}\right)$.
Then $\ln u_n \sim \frac{n^{1-\alpha}}{1-\alpha}$, so
that $\limti n \frac{\ln u_n}{\ln n} = +\infty$. In other words,
$A$ has zero exponential density.

\textbf{Acknowledgement:} A large part of this research was carried out during stays of the first named author
in Bratislava, on invitation by the Slovak University of Technology in Bratislava (Slovensk\'a Technick\'a
Univerzita v Bratislave) and the Slovak Academy of Sciences (Slovensk\'a Akad\'emia Vied).
The second author was supported by grants VEGA/1/0588/09 and UK/374/2009.
The third author was supported by grant VEGA/2/7138/27.

\end{document}